\documentclass[12pt]{amsart}
  \usepackage{latexsym} 
  \usepackage[all]{xy}
  \usepackage{amsfonts} 
  \usepackage{amsthm} 
  \usepackage{amsmath} 
  \usepackage{amssymb}
  \usepackage{pifont}  
  
  \usepackage[T1]{fontenc}
\normalfont

  \def\<{{\langle}} 
  \def\>{{\rangle}}

  \def\eps{\varepsilon}

  \def\note#1{{}} 
 \def\can{{\rm \sf{can}}} 

 \def\graph#1{{\rm \sf{graph}}(#1)} 

  \def\note#1{}

  \def\cE{{\mathfrak E}}

   \def\cB{{\mathfrak B}}

  \def\rhom#1#2#3{{{\rm Hom}\sb{#1}(#2,#3)}}

  \def\beq{\begin{equation}} 
  \def\eeq{\end{equation}}

  \def\Desc{{\bf Desc}}

  \def\id{\mathrm{id}}

  \def\ot{{\otimes}}

    \def\vect{\mathfrak{Vect}}
     \def\1{\mathbf{1}}
     \def\coten#1{\Box_{#1}} 
\def\Set{\mathfrak{Set}}
\def\Com#1{\mathfrak{Comod}({#1})}
\def\Mod#1{\mathfrak{Mod}({#1})}
\def\Bicom#1{\mathfrak{Bicom}({#1})}
\def\Pr#1{\mathfrak{Pr}({#1})}

\def\k{\Bbbk}
  \def\Desc{\mathfrak{Desc}}

  \newcounter{zlist} 
  \newenvironment{zlist}{\begin{list}{(\arabic{zlist})}{ 
  \usecounter{zlist}\leftmargin2.5em\labelwidth2em\labelsep0.5em 
  \topsep0.6ex
  \parsep0.3ex plus0.2ex minus0.1ex}}{\end{list}}

  \newcounter{blist} 
  \newenvironment{blist}{\begin{list}{(\alph{blist})}{ 
  \usecounter{blist}\leftmargin2.5em\labelwidth2em\labelsep0.5em 
  \topsep0.6ex 
  \parsep0.3ex plus0.2ex minus0.1ex}}{\end{list}} 

  \newcounter{rlist}

   \newcounter{alist} 
  \newenvironment{alist}{\begin{list}{(\Alph{alist})}{ 
  \usecounter{alist}\leftmargin2.5em\labelwidth2em\labelsep0.5em 
  \topsep0.6ex 
  \parsep0.3ex plus0.2ex minus0.1ex}}{\end{list}}

\def\stac#1{\raise-.2cm\hbox{$\stackrel{\displaystyle\otimes}{\scriptscriptstyle{#1}}$}}

\def\cten#1{\raise-.2cm\hbox{$\stackrel{\displaystyle\widehat{\otimes}}
{\scriptscriptstyle{#1}}$}}

  \def\Label#1{\label{#1}\ifmmode\llap{[#1] }\else 
  \marginpar{\smash{\hbox{\tiny [#1]}}}\fi} 
  \def\Label{\label}

  \newtheorem{proposition}{Proposition}[section]
  \newtheorem{lemma}[proposition]{Lemma}

   \newtheorem{lefinition}[proposition]{Lemma \& Definition}
  \newtheorem{ident}[proposition]{Identification}

  \theoremstyle{definition} 
  \newtheorem{definition}[proposition]{Definition}

  \theoremstyle{remark}

  \theoremstyle{definition} 
  \swapnumbers

\begin{document} 

 \title{On synthetic interpretation of quantum principal bundles} 
 \author{Tomasz Brzezi\'nski}
 \address{ Department of Mathematics, Swansea University, 
  Singleton Park, 
 \newline\indent  
  Swansea SA2 8PP, U.K.} 
  \email{T.Brzezinski@swansea.ac.uk}   
    \date{November 2009} 
  \begin{abstract} 
Quantum principal bundles or principal comodule algebras are re-interpreted as principal bundles within a framework of Synthetic Noncommutative  Differential Geometry. More specifically, the notion of a noncommutative principal bundle within a braided monoidal category is introduced and it is shown that a noncommutative principal bundle in the category opposite to the category of vector spaces is the same as a faithfully flat Hopf-Galois extension. 
  \end{abstract} 
  \maketitle

\section{Introduction}
The idea of Synthetic Differential Geometry \cite{Koc:syn} which originates from synthetic considerations of Sophus Lie is very simple: All geometric constructions are performed within a suitable base category in which space forms are objects. For classical geometry the base category is cartesian closed, i.e.\ it has finite products (e.g.\ the cartesian product in the category of sets) and the exponentials (that is the collection of functions between space forms is such a form). Furthermore it often comes equipped with a method of constructing and classifying sub-objects (the sub-object classifier); in a word a base category is a {\em topos}. In this article we use a synthetic method in description or re-interpretation of principal bundles in Noncommutative Differential Geometry.

This article is addressed to noncommutative geometers and Hopf algebraists who are familiar with general category theory culture but are not experts in category theory. The aim is to explain categorical ingredients that enter a synthetic definition of a principal bundle and then to show that noncommutative generalisation of this definition yields in particular principal comodule algebras or faithfully flat Hopf-Galois extensions. The full geometric potential of Hopf-Galois extensions was probably first explored in \cite{BrzMaj:gau}, but only now solid topological and geometric evidence is being gathered that affirms that principal comodule algebras should be accepted as principal bundles in noncommutative geometry \cite{BauHaj:pet}. Thus this article might be intrepeted as making a synthetic case for this claim.  We also hope that this article will contribute to greater appreciation of the power of categorical (and synthetic) thinking in Noncommutative Geometry. 

We assume that the reader is familiar with basic concepts and terminology of category theory such as a monomorphism, epimorphism, functor, natural transformation, adjoint functors, unit and counit of adjunction and equivalence of categories. Some familiarity with finite limits, in particular products, pullbacks and equalisers is also needed, although these notions will be explained in due course. \cite{Mac:cat} is the standard reference for all categorical concepts not explained in this article. Details of Hopf algebra theory in braided monoidal categories that feature prominently in Section~\ref{sec.quant} may be found in \cite[Chapter~9]{Maj:fou}. 

\section{Explaining the definition of principal bundles}\label{sec.clas}
\setcounter{equation}{0}
In a synthetic formulation the definition of a principal bundle is very succinct; see \cite{Koc:fib}, \cite{Koc:alg}. Let $\cE$ be a category with finite limits (for example the category $\Set$ of sets or the category of smooth manifolds). This means that $\cE$ has a terminal object, binary products and equalisers; see \cite[Section~V.2, ~ Corollary~1]{Mac:cat}. We briefly explain these three terms presently.
A {\em terminal object} in $\cE$ is an object that is a codomain for exactly one arrow from any object in $\cE$ (for example, any singleton set is a terminal object in $\Set$). A {\em (binary) product} of objects  $E_1$, $E_2$ is an object $E_1\times E_2$ together with two arrows $p_1: E_1\times E_2\to E_1$ and $p_2: E_1\times E_2\to E_2$, known as {\em projections}, such that, for any morphisms $q_1: D\to E_1$, $q_2:D\to E_2$, there exists a {\bf unique} morphism $f: D\to E_1\times E_2$, such that $p_1\circ f = q_1$ and $p_2\circ f =q_2$ (for example the cartesian product of sets is a product in $\Set$).   An {\em equaliser} of morphisms  $q_1, q_2: D\to E$ is an object $C$ together with a morphism $\iota : C\to D$ that {\em equalises} $q_1$ and $q_2$, that is $q_1\circ \iota = q_2\circ \iota$, and has the following {\em universal property}. For any morphism $p: B\to D$ that equalises $q_1$ and $q_2$, there exists a {\bf unique} arrow $q: B\to C$ such that $p = \iota\circ q$. 

Let $G$ a group object in $\cE$ (i.e. an object with associative multiplication $G\times G\to G$, neutral (global) element $* \to G$, where $*$ is  a terminal object in $\cE$, and with the inverses given by an isomorphism $G\to G$). In case of category of sets $G$ is simply a group, in the category of smooth manifolds it is a Lie group etc. Let $B$  be an object in $\cE$. A principal $G$-bundle over $B$ is an object $P$ together with a morphism $\pi: P\to B$ and a (right) $G$-action $\varrho : P\times G \to P$ such that the the following diagram
\begin{equation}\label{diag1}
\xymatrix{ P\times G
\ar@<0.5ex>[r]^-{p}\ar@<-0.5ex>[r]_-{\varrho} & 
P \ar[r]^\pi &B ,}
\end{equation}
where $p$ is the projection, is exact. The term {\em exact} means that
\begin{blist}
\item \eqref{diag1} is a coequaliser diagram,
\item \eqref{diag1} is a kernel pair diagram,
\item $\pi$ is an effective descent epimorphism,
\end{blist}
and our aim in this section is to explain the meaning of (a)--(c).

The meaning of (a) is most straightforward as the definition of a coequaliser is obtained by reversing of all the arrows in already recalled definition of an equaliser. Explicitly, \eqref{diag1} is a coequaliser diagram provided $\pi$ {\em coequalises} $p$ and $\varrho$, i.e.,
$
\pi\circ p = \pi\circ \varrho,
$
and for any other morphism $\phi: P\to A$ coequalising $p$ and $\varrho$ there exists a unique morphism $\psi: B\to A$ fitting the following diagram
$$
\xymatrix{ P\times G
\ar@<0.5ex>[r]^-{p}\ar@<-0.5ex>[r]_-{\varrho} & 
P \ar[r]^\pi\ar[d]_\phi &B\ar[ld]^\psi \\
& A.&}
$$
The existence of such a unique $\psi$ is referred to as the {\em universal property of coequalisers}.

Before moving to condition (b), we can look at the meaning of (a) in case $\cE$ is the category $\Set$ of sets. In this case $P\times G$ is the usual cartesian product, $p$, $\varrho$, $\pi$ are functions and writing $x\cdot g$ for $\varrho(x,g)$ (the action of $g\in G$ on $x\in P$), the coequalising property of $\pi$ means that, for all $b\in G, x\in P$,
$$
\pi(x\cdot g) = \pi\left(p\left(x,g\right)\right) = \pi(x).
$$
The universal property now means that (up to an isomorphism) $B = P/G$ (the quotient set) and, by the coequalising property, $\pi: P\to B$ is the canonical surjection. 

Since $\cE$ has finite limits, in particular it has pullbacks. That is, for any pair of morphisms with common codomain, $f_1: E_1 \to B$, $f_2: E_2 \to B$, there is a 
diagram (i.e.\ an object $E_1\times_B E_2$ and two morphisms $p_1$, $p_2$ fitting the following diagram)
$$
\xymatrix{ E_1 \times_B E_2 \ar[r]^-{p_1} \ar[d]_{p_2} & E_1 \ar[d]^{f_1} \\
E_2 \ar[r]_{f_2} & B,}
$$ 
with the following universal property. For any object $F$ and morphisms $q_1: F\to E_1$ and $q_2: F\to E_2$ such that 
$
f_2 \circ q_2 = f_1 \circ q_1$, there exists a unique morphism $\gamma : F\to E_1 \times_B E_2$ such that $p_2\circ \gamma = q_2$ and $p_1\circ \gamma = q_1$. One should keep in mind that a pullback (as every finite limit) is a particular equaliser; see \cite[Section~V.2, ~ Theorem~2]{Mac:cat}.

The statement (b) that  \eqref{diag1} is a kernel pair diagram means that the (unique) map $\gamma: P\times G \to P\times_B P$,
$$
\xymatrix{ P\times G \ar@{-->}[rd]^-\gamma \ar@/^/[drr]^-p\ar@/_/[ddr]_-\varrho && \\ & P \times_B P \ar[r]^-{p_1} \ar[d]_{p_2} & P \ar[d]^\pi \\
& P \ar[r]_\pi & B,}
$$ 
is an isomorphism. 

Again we can look at the case of sets. The pullback $P\times_B P$ of $\pi$ is a subset of $P\times P$,
$$
P\times_B P = \{ (x,y) \in P\times P\; |\; \pi(x)=\pi(y)\},
$$
while $p_1, p_2$ are restrictions of canonical projections $P\times P \to P$. The map $\gamma$ is therefore given by
$$
\gamma(x,g) = (x, x\cdot g),
$$
and its bijectivity means precisely that the action of $G$ on $P$ is free. 

The explanation of the condition (c) requires a quick tour of the descent theory. We follow formulation of effective descent morphisms presented in \cite{JanSob:bey}. With any object $A$ of $\cE$ one can associate its {\em comma category} or {\em category of objects over $A$},  $(\cE \downarrow A)$. Objects in $(\cE \downarrow A)$ are morphisms with codomain $A$, i.e.\ pairs $(E,f)$, where $E$ is an object in $\cE$ and $f:E\to A$ is a morphism. A morphism from $(E_1,f_1)$ to $(E_2,f_2)$ in the comma category $(\cE \downarrow A)$ is a morphism $h: E_1\to E_2$ in $\cE$ making the following triangle
\begin{equation}\label{triang}
\xymatrix{E_1 \ar[rr]^-h\ar[rd]_{f_1} && E_2\ar[ld]^{f_2} \\ & A &}
\end{equation}
commute. While objects in $\cE$ have no elements in a set theoretic sense, one can treat all morphisms with codomain $A$ as {\em generalised elements} of the object $A$. Thus  $(\cE \downarrow A)$ can be thought of as a collection of elements of $A$; see \cite[Part~II]{Koc:syn} for an in-depth discussion of generalised elements in relation to comma categories.

Any morphism $\pi: P\to B$ in $\cE$ induces a pair of functors between the comma categories $(\cE \downarrow P)$ and $(\cE \downarrow B)$. The functor  $\pi_!: (\cE \downarrow P)\to (\cE \downarrow B)$ is given by the composition with $\pi$, that is
$\pi_!(E,f) = (E, \pi\circ f)$ on objects and $\pi_!(h) =h$ on morphisms in $(\cE \downarrow P)$. 
The functor $\pi^*: (\cE \downarrow B)\to (\cE \downarrow P)$, to every morphism $f: E\to B$ in $\cE$ assigns the projection $p_2$ in the pullback
\begin{equation}\label{pullback}
\xymatrix{ E \times_B P \ar[r]^-{p_1} \ar[d]_{p_2} & E \ar[d]^f & \\
P \ar[r]^\pi & B, & \qquad \pi^*(E,f) = (E\times_B P, p_2).}
\end{equation}
For all morphisms $h: (E_1, f_1) \to (E_2,f_2)$ in $(\cE \downarrow B)$, $\pi^*(h)$ is a unique filler in the following pullback diagram
$$
\xymatrix{ E_1\times_B P \ar@{-->}[rd]^{\!\!\!\!\pi^*(h)} \ar@/^/[drr]^-{h\circ p_1}\ar@/_/[ddr] && \\ & E_2 \times_B P \ar[r] \ar[d] & E_2 \ar[d]^{f_2} \\
& P \ar[r]_\pi & B,}
$$ 
where $p_1$ and unmarked arrows are relevant canonical projections. 

The functor $\pi_!$ is the left adjoint of the functor $\pi^*$. The counit of this adjunction, $\epsilon: \pi_!\circ \pi^* \to \id_{(\cE \downarrow B)}$, assigns to a morphism $f:E\to B$ the projection $p_1$ in the pullback \eqref{pullback}, which is clearly a morphism in $(\cE \downarrow B)$ between 
$$
\pi_!\circ \pi^*(E,f) = \pi_! (E\times_B P, p_2) = (E\times_B P, \pi\circ p_2)
$$
and $(E,f)$. The unit $\eta: \id_{(\cE \downarrow P)} \to \pi^*\circ \pi_!$ assigns to $f: E\to P$ the unique morphism $\langle E,f\rangle$ that fits into the diagram
$$
\xymatrix{ E \ar@{-->}[rd]^{\!\!\langle E,f\rangle} \ar@/^/[drrr]^-{\id_E}\ar@/_/[ddr]_-{f} &&& \\ & E \times_B P \ar[rr]^-{p_1} \ar[d]_{p_2} && E \ar[d]^{f\circ\pi} \\
& P \ar[rr]_\pi && B.}
$$ 

A {\em monad} or a {\em triple} on a category $\cE$ is a functor $T: \cE\to\cE$ together with two natural transformations $\mu: T\circ T \to T$ (the multiplication) and $\eta: \id_\cE\to T$ (the unit) which satisfy the associativity and unitality conditions. Explicitly, recall that to each object $E$ in $\cE$, a natural transformation $\mu$ assigns  a morphism $\mu_E : T\circ T(E) \to T(E)$ and $\eta$ assigns  a morphism $\eta_E: E\to T(E)$. The associativity and unitality conditions state that, for all objects $E$,
$$
\mu_E \circ \mu_{T(E)} = \mu_E\circ T(\mu_E), \quad \mu_E\circ \eta_{T(E)} = \mu_E\circ T(\eta_{E}) = \id_{T(E)}.
$$
To a monad $(T,\mu,\eta)$ one associates its representation category known as the Eilenberg-Moore category of {\em algebras} or {\em modules} $\cE^T$. The objects are pairs $(X,\xi)$, where $X$ is an object in $\cE$ and $\xi: T(X)\to X$ is a morphism such that
$$
\xi\circ \mu_X = \xi\circ T(\xi), \qquad \xi\circ \eta_X = \id_X,
$$
(i.e.\ $\xi$ is associative and unital). One refers to $\xi$ as a {\em structure map} for the module $X$. A morphism of modules is a morphism in $\cE$ that commutes with the structure maps.

We have all the ingredients needed for explaining condition (c) at hand. Going back to the morphism $\pi: P\to B$, the adjunction $\pi_! \dashv \pi^*$ induces a monad $T = \pi^*\circ \pi_!$ on $(\cE \downarrow P)$ with $\mu_{(E,f)} = \pi^*(\epsilon_{\pi_!(E,f)}) = \pi^*(\epsilon_{(E,\pi\circ f)})$, for all $(E,f) \in (\cE \downarrow P)$, and unit $\eta$. Furthermore, it induces the functor $K: (\cE \downarrow B)\to (\cE \downarrow P)^T$, known as the {\em comparison functor}, that fits the following diagram
$$
\xymatrix{ (\cE \downarrow B)\ar[rr]^{K}\ar[dr]_{\pi^*} & & (\cE \downarrow P)^T\ar[dl]^{U} \\
& (\cE \downarrow P)& ,}
$$
where $U$ is the forgetful functor. The commutativity of the diagram means that the object part of $K(E,f)$ is the projection $p_2: E \times_B P \to P$ in the pullback diagram \eqref{pullback}, i.e.\ $\pi^*(E,f)$. The structure map is 
$$
\pi^*(\epsilon_{(E ,f)}) : T\circ \pi^*(E,f) = \pi^*\circ \pi_!\circ \pi^*(E,f) \longrightarrow \pi^*(E,f) = (E \times_B P, p_2).
$$
The morphism $\pi: P\to B$ is an {\em effective descent morphism} provided the associated comparison functor $K$ is an equivalence of categories. Of course effective descent epimorphism is an effective descent morphism which is an epimorphism. 

In case of the category of sets, every epimorphism (surjective function) is an effective descent morphism, but in case of other categories condition (c) carries non-trivial information.

\section{Translating the definition of principal bundles}\label{sec.quant} \setcounter{equation}{0}
In this section we would like to translate carefully the definition of principal bundles to the case of braided monoidal categories with equalisers preserved by the tensor product. The ongoing programme whose foundations are presented in \cite{Mas:non} makes a very solid case for such categories to be a proper environment for Noncommutative Geometry. 

A category $\cB$ is called
a {\em  monoidal category}  if there exist
 a functor $-\ot-   :\cB \times\cB \to \cB $, a distinguished
 object $\1$ in $\cB$ and isomorphisms
 $$\alpha_{X,Y,Z} : (X\ot Y)\ot Z \to X\ot (Y\ot Z), \quad  
  \lambda_X: \1\ot X \to X,\quad \varrho_X:X\ot \1\to X,$$ 
natural in $X,Y,Z$, such that, for all objects $W,X,Y,Z$ in $\cB $, the following diagrams
$$\xymatrix{
((W\ot X)\ot  Y)\ot  Z\ar[rr]^-{\alpha_{W,X,Y}\ot  \id_Z}\ar[d]_{\alpha_{W\ot 
X,Y,Z}} & &  (W\ot  (X\ot  Y))\ot  Z \ar[rr]^-{\alpha_{W, X\ot Y,Z}} && W\ot  ((X\ot  Y)\ot  Z)\ar[d]^{\quad \id_W\ot \alpha_{X,Y,Z}} \\ 
    (W\ot  X)\ot  (Y\ot  Z)  \ar[rrrr]^-{\alpha_{W,X,Y\ot Z}}  
 & & & & W\ot  (X\ot  (Y\ot Z))\, ,  }$$ 
and 
$$\xymatrix{
(X\ot \1)\ot Y\ar[dr]_{\varrho_{X}\ot \id_Y}\ar[rr]^{\alpha_{X,\1,Y}} & & X\ot (\1\ot Y)
     \ar[dl]^{\id_X\ot\lambda_Y} &\\ 
 & X\ot Y &  }
 $$ 
 commute.  
 The functor $\ot$ is referred to as the {\em tensor} or {\em mnoidal product} and $\1$ is called the {\em monoidal unit}. A monoidal category is said to be {\em strict monoidal} if the isomorphisms $\alpha_{X,Y,Z}$, $\lambda_X$ and $\gamma_X$ are identities, for all objects $X,Y,Z$ of $\cB$. Although there might be rather few strict monoidal categories in nature, any monoidal category is equivalent to a strict one. It is therefore customary (and very convenient) to treat monoidal category as if it were strict monoidal, and not to write the rearranging brackets isomorphisms (or {\em associators}) $\alpha_{X,Y,Z}$ and the {\em unitors} $\lambda_X$ and $\varrho_X$ in any formulae (each of  these Greek letters will denote something else in the sequel). With this convention in mind, a monoidal category is denoted by $(\cB, \otimes, \1, \tau)$, and the brackets are not written between multiple tensor products.
 
 A (strict) monoidal category  $(\cB, \otimes, \1, \tau)$ is said to be a {\em (strict) braided monoidal category} if, for any pair of objects $X,Y$ of $\cB$, there exists an isomorphism, called a {\em braiding}, $\tau_{X,Y}: X\ot Y \to Y\ot X$, natural  in $X$ and $Y$, such that, for all objects $X,Y,Z$ of $\cB$, $\tau_{\1,X}$, $\tau_{X,\1}$ are identity morphisms (up to unitors), and 
 $$
 (\id_Y\ot \tau_{X,Z})\circ (\tau_{X,Y}\ot \id_Z) = \tau_{X, Y\ot Z},  \qquad (\tau_{X,Z}\ot\id_Y)\circ(\id_X\ot \tau_{Y,Z}) = \tau_{X\ot Y, Z},
 $$
 (the associators in the first equality and their inverses in the second one are not written explicitly). The naturality of braiding means that, for all morphisms $f: X\to Y$ and all objects $Z$,
 $$
 \tau_{Y,Z}\circ(f\ot\id_Z) = (\id_Z\ot f)\circ\tau_{X,Z}, \qquad \tau_{Z,Y}\circ(\id_Z\ot f) = (f\ot \id_Z)\circ\tau_{Z,X}.
 $$
A braided monoidal category 
is denoted by $(\cB ,\ot, \1,\tau)$. 
The examples of particular interest are the category $\vect_\k$ of  vector spaces over a field $\k$, with the monoidal product given by the standard tensor product of $\k$-vector spaces, monoidal unit $\1 =\k$ and with the braiding $\tau$ provided by the flip, and its opposite category $\vect_\k^{op}$. (Recall that the opposite category is obtained from the original one by reversing all the arrows.) 

As in the category of vector spaces in any monoidal category one can consider {\em monoids} or {\em algebras} and {\em comonoids} or {\em coalgebras}. These are defined in the same diagrammatic way as usual $\k$-algebras and  $\k$-coalgebras. In braided monoidal category, the tensor product of two coalgebras (algebras) is again a coalgebra (algebra). If $C$ is a coalgebra in $(\cB ,\ot, \1,\tau)$ with comultiplication $\Delta_C:C\to C\ot C$ and counit $\eps_C: C\to \1$, and $D$ is a coalgebra with comultiplication $\Delta_D:D\to D\ot D$ and counit $\eps_D: D\to \1$, then $C\ot D$ is a coalgebra with counit $\eps_C\ot \eps_D$ and comultiplication
\begin{equation}\label{eq.coalg}
\xymatrix{C\ot D \ar[rr]^-{\Delta_C\ot \Delta_D} && C\ot C\ot D\ot D \ar[rr]^-{\id_C\ot \tau_{C,D} \ot \id_D} &&  C\ot D\ot C\ot D.}
\end{equation}
Similarly, if  $A$ is an algebra with multiplication $m_A$ and unit $1_A: \1 \to A$, and $B$ is an algebra with multiplication $m_B$ and unit $1_B: \1\to A$, then $A\ot B$ is an algebra with unit $1_A\ot 1_B$ and multiplication
\begin{equation}\label{eq.alg}
\xymatrix{A\ot B\ot A\ot B \ar[rr]^-{\id_A\ot \tau_{B,A} \ot \id_B} && A\ot A\ot B\ot B \ar[rr]^-{m_A\ot m_B} && A\ot B.}
\end{equation}
Thus, similarly as for $\k$-algebras and $\k$-coalgebras, in a braided monoidal category one can combine algebras and coalgebras to define bialgebras and Hopf algebras.  

The aim of this section is to translate the definition of principal bundles to the case of braided monoidal categories. The translation procedure we employ follows the same principle as in \cite{Agu:int}: we interpret the category of sets as a monoidal category (with the cartesian product understood as the tensor product, the monoidal unit given by the singleton set, and the braiding provided by the flip) and  reformulate notions described in the preceding section in the way suitable for a general monoidal category. As with every translation, some information is lost some other gained, but it is hoped that the general spirit is preserved. 

The first observation is that every set $X$ is a {\em comonoid} or a {\em coalgebra} with the unique comultiplication given by the diagonal function,
$$
\Delta_X : X\to X\times X, \qquad x\mapsto (x,x).
$$
The coassociativity of this comultiplication is obvious. The counit is the unique map $\eps_X : X \to *$, where $*$ is a fixed singleton set understood as the monoidal unit in the category of sets. The standard identification $X\times * \cong X \cong *\times X$ immediately yields the counitality of $\Delta_X$. Guided by this we should consider comonoids or coalgebras  in $(\cB, \otimes, \1, \tau)$ as basic objects. 

Since any set has unique comultiplication and counit, a group $G$ is the same as a Hopf algebra in the category of sets; the standard compatibility conditions between comultiplication and multiplication etc., are automatically satisfied (the singleton set $*$ is understood as a monoid in the only possible way). The antipode $S$ is given by the inverse function
$$
S : G \to G , \qquad g\mapsto g^{-1}.
$$
This antipode is obviously bijective (in fact, involutive, since $S\circ S =\id_G$). Thus for a group object in a braided category $\cB$ we take a Hopf algebra $H$ in $\cB$ (and we may assume that $H$ has a bijective antipode, although it is not needed at this point). 

Again, since the comonoid structure on a set is uniquely given, the action of a group $G$ on a set $P$ is compatible with the comonoid structures. This means that the action $\varrho : P\times G\to P$ makes $P$ into a right module coalgebra over the Hopf algebra $G$ (all in the category of sets), i.e., for all $x\in P$ and $g\in G$,
\begin{equation}\label{mod.coa}
\Delta_P(x\cdot g) = \Delta_P(x)\cdot \Delta_G(g),
\end{equation}
where the action on the right hand side is defined componentwise. Obviously $\eps_P(x\cdot g) = \eps_P(x)\eps_G(g)$. Thus, in a braided monoidal category $\cB$, $P$ is required to be a right $H$-module coalgebra with the $H$-action $\varrho: P\ot H \to P$. The compatibility \eqref{mod.coa} is then expressed in the element-free way and with the use of the braiding $\tau$ as the commutativity of the following diagram
$$
\xymatrix{P\otimes H \ar[r]^-\varrho \ar[d]_{\Delta_P\ot \Delta_H} & P\ar[r]^-{\Delta_P} & P\ot P \\
P\ot P\ot H\ot H \ar[rr]^-{\id_P\ot \tau_{P,H}\ot \id_H} && P\ot H\ot P\ot H \ar[u]_{\varrho\ot \varrho } ,}
$$
and $\eps_P\circ \varrho = \eps_P\ot \eps_H$, i.e.\ $\varrho$ is a morphism of coalgebras (comonoids) in $\cB$.

In the case of sets the canonical projection $p: P\times G \to P$, $(x,g)\mapsto x$ can be identified with the function $(x,g)\mapsto (x,*) \equiv x$ (where $*$ stands for the only element of the singleton set $*$), i.e.\ $p$ can be identified with $\id_P\times \eps_G$. In this form, the map $p$ can be translated to a braided monoidal category $\cB$.

These are  ingredients needed for the definition of a noncommutative principal bundle.

\begin{definition}\label{def.main}
Let $(\cB, \otimes, \1, \tau)$ be a braided monoidal category with equalisers preserved by the tensor product. Let $H$ be a Hopf algebra in $\cB$ and $B$ be a coalgebra in $\cB$.  A right $H$-module coalgebra $P$ (with action $\varrho: P\ot H\to P$) together with a coalgebra morphism $\pi: P\to B$ is called a {\em noncommutative $H$-principal bundle} provided that the following diagram
\begin{equation}\label{diag2}
\xymatrix{ P\otimes H
\ar@<0.5ex>[rr]^-{\id_P\ot \eps_H}\ar@<-0.5ex>[rr]_-{\varrho} && 
P \ar[r]^\pi &B ,}
\end{equation}
of comonoids is $\ot$-exact.
\end{definition}
The term `$\ot$-exact' is a translation of the term `exact' explained in Section~\ref{sec.clas} to braided monoidal categories. Therefore,  its  meaning comprises three statements:
\begin{alist}
\item  \eqref{diag2} is a coequaliser diagram,
\item \eqref{diag2} is a kernel pair diagram of comonoids in $\cB$,
\item  $\pi$ is an effective descent epimorphism of comonoids in $\cB$.
\end{alist}
The statement (A) is the same as (a) in Section~\ref{sec.clas}, the meaning of (B) and (C) need to be explained. In due course the need for and meaning of the assumptions that $\cB$ has equalisers and that they are preserved by the tensor product should become clear. 

As already mentioned in Section~\ref{sec.clas}, in the category of sets a pullback of $\alpha: E_1 \to B$ and $\beta: E_2 \to B$ is a subset of $E_1\times E_2$ defined by
$$
E_1\times_B E_2 = \{(x,y)\in E_1\times E_2\; |\; \alpha(x) = \beta(y)\}.
$$
Since all sets are comonoids, and functions are maps of comonoids, the set $E_1$ is a right $B$-comodule and $E_2$ is a left $B$-comodule with coactions
$$
\lambda_1 = (\id_{E_1}\times \alpha)\circ \Delta_{E_1} : E_1\to E_1\times B, \qquad x\mapsto (x,\alpha(x)),
$$
and 
$$
\lambda_2 = (\beta\times \id_{E_2})\circ \Delta_{E_2} : E_2\to B\times E_2, \qquad y\mapsto (\beta(y),y).
$$
Thus
$$
E_1\times_B E_2 
= \{(x,y)\in E_1\times E_2\; |\; (x,\alpha(x),y) = (x,\beta(y),y) \} = E_1\coten B E_2,
$$
where $E_1\coten B E_2$ denotes the equaliser of $\lambda_1\times \id_{E_2}$ and $\id_{E_1}\times \lambda_2$, i.e.\ the {\em cotensor product} of comodules. This indicates that pullbacks in a category with finite limits should be translated to cotensor products of comodules in a monoidal category. 

While the decision which of the monoids turn into left and which one in the right comodule is somewhat arbitrary, this ambiguity vanishes if we want to equalise the same morphism.  This leads to the following lemma and definition which explain (B).
\begin{lefinition}\label{lem.can}
Let $\cB$ be a monoidal category (not necessarily braided). Consider a diagram of comonoids
\begin{equation}\label{diag3}
\xymatrix{ C
\ar@<0.5ex>[r]^-{\alpha}\ar@<-0.5ex>[r]_-{\beta} & 
P \ar[r]^\pi &B ,}
\end{equation}
in which $\pi\circ\alpha = \pi\circ\beta$, i.e.\ $\pi$ coequalises $\alpha$ and $\beta$.
Assume that there exists the equaliser of $(\id_P\ot \pi \ot \id_P)\circ (\Delta_P\ot \id_P)$ and $(\id_P\ot \pi \ot \id_P)\circ (\id_P\ot \Delta_P)$, and denote it by  $P\coten B P$. Then there exists  a unique morphism
$$
\can: C \to P\coten B P,
$$
such that
$$
\xymatrix{C\ar[rr]^-\can \ar[rrd]_{(\alpha \ot \beta)\circ \Delta_C} &&  P\coten B P\ar[d]\\
&& P\ot P.}
$$
The diagram \eqref{diag3} is said to be a {\em kernel pair  diagram of comonoids in $\cB$} provided the map $\can$ is an isomorphism.
\end{lefinition}
\begin{proof}
Start with the following  straightforward calculation:
\begin{eqnarray*}
(\id_P\ot \pi \ot \id_P)\circ (\Delta_P\ot \id_P)\circ (\alpha\ot \beta )\circ \Delta_C &= &(\id_P\ot \pi \ot \id_P)\circ  (\alpha\ot\alpha\ot\beta)\circ (\Delta_C\ot \id_C)\circ\Delta_C\\
&=& (\id_P\ot \pi \ot \id_P)\circ  (\alpha\ot\beta \ot \beta)\circ (\id_C\ot \Delta_C)\circ\Delta_C\\
&=& (\id_P\ot \pi \ot \id_P)\circ (\id_P\ot \Delta_P)\circ (\alpha\ot \beta)\circ \Delta_C .
\end{eqnarray*}
The first and the third equalities follow by the fact that $\alpha$ and $\beta$ are morphisms of comonoids, while the second equality is a consequence of the coassociativity of $\Delta_C$ and the fact that $\pi$ coequalises $\alpha$ and $\beta$. The existence and uniqueness of the morphism $\can$ now follow by the universal property of equalisers.
\end{proof}

\begin{lemma}\label{lem.can.H}
In the setting of Definition~\ref{def.main}, 
$\can: P\ot H \to P\coten B P$ is the unique morphism induced by the composite
\begin{equation}\label{eq.can}
\xymatrix{ P\ot H \ar[rr]^{\Delta_P\ot \id_H} && P\ot P\ot H \ar[rr]^-{\id_P \ot \varrho} && P\ot P.
}
\end{equation}
\end{lemma}
\begin{proof} The form of the map $\can$ is obtained by the following computation 
\begin{eqnarray*}
(\id_P \ot \eps_H \ot \varrho)\circ\Delta_{P\ot H} &=& (\id_P \ot \eps_H \ot \varrho)\circ (\id_P\ot \tau_{P,H} \ot \id_H)\circ(\Delta_P\ot \Delta_H) \\
&=&  (\id_P \ot \varrho)\circ(\id_P \ot \id_P\ot \eps_H\ot \id_H) \circ(\Delta_P\ot \Delta_H) \\
&=& (\id_P \ot \varrho)\circ(\Delta_P\ot \id_H),
\end{eqnarray*}
where the second equality follows by the naturality of the braiding and from the fact that (up to unitors) the braiding $\tau_{P,\1}$ is the identity. 
\end{proof}

The property (B) thus states that the map $\can$ induced by \eqref{eq.can} is an isomorphism.

It remains to explain the meaning of (C), or, more precisely to translate (c) so that the meaning to (C) can be given. Take a set $A$. The comma category $(\Set \downarrow A)$ consists of functions $f:X\to A$. A function $f:X\to A$ can be equivalently described as a function assigning to $X$ the graph of $f$, i.e.\ as
$$
\graph{f}: X\to X\times A, \qquad x\mapsto (x, f(x)).
$$
If $A$ is understood as a comonoid (in a unique way), then $(X, \graph{f})$ is a right $A$-comodule ($\graph{f}$ is a coaction). The commutativity of the diagram \eqref{triang} for a morphism $g:(X_1, f_1)\to (X_2, f_2)$ in $(\Set \downarrow A)$ is equivalent to the commutativity of the following diagram for the induced coactions $\graph{f_1}$, $\graph{f_2}$
$$
\xymatrix{X_1 \ar[rr]^-g \ar[d]_{\graph{f_1}} && X_2\ar[d]^{\graph{f_2}}\\
X_1\times A \ar[rr]^-{g\times \id_A} && X_2\times A.}
$$
Therefore, the comma category $(\Set \downarrow A)$ is the same as or isomorphic to the category of right $A$-comodules.

In a monoidal category $\cB$, by the comma category of a comonoid $A$ we understand the category of right $A$-comodules $\Com A$. We should point out at this point that the choice of right over left $A$-comodules is arbitrary. We also do not require the objects of the comma category of $A$ to be (morphisms of) comonoids, as the notion of an effective descent morphism of comonoids can already be introduced in present generality.  

A morphism of comonoids $\pi: P\to B$ induces a functor $\pi_! : \Com P\to \Com B$ by composition with the coaction. If $(E,\lambda)$ is a right $P$-comodule, then
$$
\pi_! (E,\lambda) = (E, (\id_E\ot \pi)\circ \lambda),
$$
is a right $B$-comodule. On morphisms $\pi_!(h) =h$. In the case of sets this is exactly the functor $\pi_!$ between comma categories described in Section~\ref{sec.clas}. To see this one should use the identification of a function $f$ with the  function $\graph{f}$ into its graph. 

The construction of the functor $\pi^* :  \Com B\to \Com P$ requires additional assumptions on the monoidal category $\cB$. First  we need to look at the case of sets, and study the diagram \eqref{pullback}. The pullback in diagram \eqref{pullback} is 
\begin{eqnarray*}
E\times_B P &=& \{(e,x)\in E\times P \; |\; \pi(x) = f(e)\}\\
& = & \{(e,x)\in E\times P \; |\; (e,\pi(x),x) = (\graph{f}(e), x)\} = E\coten B P,
\end{eqnarray*}
where $E$ is the right $B$-comodule with coaction $\graph{f}$ and $P$ is the left $B$-comodule with coaction $(\pi\times \id_P)\circ \Delta_P$. The $P$-coaction induced by the projection $p_2$ is
$$
\graph{p_2} : E\times_B P \to E\times_B P\times P, \qquad (e,x)\mapsto (e,x,p_2(e,x)) = (e,x,x),
$$
that is  
$
\graph{p_2} =  \id_E\times \Delta_P\mid_{E\times_BP}.$ Translating this to the monoidal category $\cB$ we conclude that, for all $B$-comodules $(E,\lambda)$,
$$
\pi^*(E,\lambda) = (E\coten B P,  \id_E\coten B\Delta_P).
$$ 
For this functor to be defined we thus need to require the existence of equalisers in $\cB$ to obtain the object part of the comodule, and these equalisers need be preserved by the tensor product (meaning: after tensoring an equaliser we again obtain the equaliser) for $\id_E\coten B\Delta_P$ to be a well-defined coaction (this explains the origin of assumptions on $\cB$ made in Definition~\ref{def.main}). More explicitly, by $\id_E\coten B\Delta_P$ we understand the unique arrow in the following diagram
\begin{equation}\label{equaliser}
\xymatrix{ (E\coten B P) \otimes P \ar[r] & E\ot P\ot P 
\ar@<0.5ex>[rrrrr]^-{(\id_E \ot \pi \ot\id_P\ot \id_P)\circ(\id_E \ot \Delta_P\ot \id_P) }\ar@<-0.5ex>[rrrrr]_-{\lambda \ot \id_P\ot \id_P} &&&&& E\ot B\ot P\ot P\\ 
& E\ot P\ar[u]_{\id_E \ot \Delta_P} &&&&&\\
& E\coten B P\ar[u]\ar@{-->}[luu]^{\id_E\coten B\Delta_P}. &&&&& }
\end{equation}
The existence and uniqueness of $\id_E\coten B\Delta_P$ with specified codomain follow by the assumption that the tensor product preserves equalisers, so that the top row in diagram \eqref{equaliser} is an equaliser.  On morphisms of right $B$-comodules $f: (E_1,\lambda_1) \to (E_2,\lambda_2)$, $\pi^*(f) = f\coten B\id_P$ is the unique filler in the following equaliser diagram
\begin{equation}\label{eq.pi*}
\xymatrix{ E_2\coten B P  \ar[r] & E_2\ot P
\ar@<0.5ex>[rrrr]^-{(\id_{E_2} \ot \pi \ot\id_P)\circ(\id_{E_2} \ot \Delta_P) }\ar@<-0.5ex>[rrrr]_-{\lambda_2 \ot \id_P} &&&& E_2\ot B\ot P\\ 
& E_1\ot P\ar[u]_{f\ot \id_P} &&&&\\
& E_1\coten B P\ar[u]\ar@{-->}[luu]^{f\coten B\id_P}. &&&& }
\end{equation}
At this point the translation stops and the development of translated notions follows its own course.

If $\cB$ has equalisers and equalisers are preserved by the tensor product, then to any morphism of comonoids $\pi: P\to B$ one can associate a pair of functors 
$$
\pi_! : \Com P\to \Com B, \qquad \pi^*: \Com B\to \Com P.
$$
This is an adjoint pair with $\pi_!$ the left adjoint of $\pi^*$. For any $P$-comodule $(E,\lambda)$, the unit of the adjunction is
$$
\eta_{(E,\lambda)}: (E,\lambda) \to \pi^*\circ\pi_!(E,\lambda)= (E\coten B P, \id_E\coten B \Delta_P), \qquad \eta_{(E,\lambda)}= \hat{\lambda},
$$
where  $E$ on the right hand side of the first equality is understood as a right $B$-comodule with the coaction  $(\id_E\ot \pi)\circ \lambda$, and $\hat{\lambda}$ is the (universally induced) unique filler in the following diagram
$$
\xymatrix{ E\coten B P \ar[r] & E\ot P 
\ar@<0.5ex>[rrrr]^-{(\id_E \ot \pi \ot\id_P)\circ(\id_E \ot \Delta_P) }\ar@<-0.5ex>[rrrr]_-{(\id_E \ot \pi \ot\id_P)\circ(\lambda \ot \id_P)} &&&& E\ot B\ot P\\ 
& E\ar[u]^\lambda\ar@{-->}[lu]^{\hat{\lambda}}. &&&& }
$$
For any $B$-comodule $(E,\lambda)$, the counit of the adjunction 
$$
\epsilon_{(E,\lambda)}: \pi_!\circ\pi^*(E,\lambda) = \left(E\coten BP, \id_E\coten B\left(\left(\id_P\ot \pi\right)\circ \Delta_P\right)\right) \to (E,\lambda), 
$$
is  the composite 
$$
 \epsilon_{(E,\lambda)} : \xymatrix{E\coten B P\ar[r] & E\ot P \ar[rr]^-{\id_E \ot  \eps_P} && E.}
$$
The morphism $\id_E\coten B\left(\left(\id_P\ot \pi\right)\circ \Delta_P\right)$ is defined by a diagram similar to \eqref{equaliser}. Since $\pi:P\to B$ is a morphism of comonoids, $\eps_B\circ \pi = \eps_P$, and $E\coten B B$ is isomorphic to $E$ (with the isomorphism induced by $\id_E\ot \eps_B$),  the counit of adjunction $ \epsilon_{(E,\lambda)}$ can be identified with $\id_E\coten B \pi$. 

The monad $(T = \pi^*\circ\pi_!,\mu,\eta)$ on the category $\Com P$ induced by the adjunction $\pi_!\dashv \pi^*$ according to the standard procedure described in Section~\ref{sec.clas} comes out as follows. For any right $P$-comodule $(E,\lambda)$,
\begin{equation}\label{T}
T(E,\lambda) = (E\coten B P, \id_E\coten B \Delta_P), \qquad \eta_{(E,\lambda)} = \hat{\lambda},
\end{equation}
and 
\begin{equation}\label{mu}
\mu_{(E,\lambda)} =  \epsilon_{(E,\lambda)} \coten B \id_P:  \xymatrix{E\coten B P\coten B P \ar[r] & (E\ot P)\coten BP \ar[rrr]^-{(\id_E \ot  \eps_P)\coten B \id_P} &&& E\coten B P,}
\end{equation}
where $E$ is understood as a right $B$-comodule by the coaction acquired from $\pi_!$, i.e., $(\id_E\ot \pi)\circ \lambda$. 

An algebra or module over the monad $(T,\mu,\eta)$ is a triple $(E,\lambda,\xi)$, in which $(E,\lambda)$ is a right $P$-comodule (i.e.\ an object in the category $\Com P$ on which $T$ operates) and $\xi: E\coten B P \to E$ is a morphism of $P$-comodules rendering commutative the following diagrams
\begin{equation}\label{alg}
\xymatrix{ E\coten B P \coten BP\ar[rr]^-{\mu_{(E,\lambda)}} \ar[d]_{\xi \coten B \id_P} && E\coten B P \ar[d]^\xi \\
E\coten B P \ar[rr]^-\xi && E,}\qquad 
\xymatrix{E  \ar[rr]^-{\hat{\lambda}}\ar[rd]_{\id_E} && E\coten B P\ar[ld]^\xi \\
& E. &}
\end{equation}
As in Section~\ref{sec.clas} (or, indeed, as in the case of any adjoint pair of  functors), associated to the comonad morphism $\pi: P\to B$ there is a comparison functor $K$ connecting the category of $B$ comodules with the category  $ \Com P^T$ of algebras over the monad $T=\pi^*\circ\pi_!$. $K$ fits into the commutative triangle
$$
\xymatrix{ \Com B\ar[rr]^{K}\ar[dr]_{\pi^*} & & \Com P^T\ar[dl]^{U} \\
& \Com P.& }
$$
Explicitly, for any $B$-comodule $(E,\lambda)$,
\begin{equation} \label{comparison}
K(E,\lambda) = \left( E\coten B P ,\id_E\coten B\Delta_P , \epsilon_{(E,\lambda)}\coten B \id_P\right).
\end{equation}
An epimorphism of comonoids $\pi: P\to B$ in the monoidal category $\cB$ (with equalisers preserved by the tensor product) is called an {\em effective descent epimorphism of comonoids in $\cB$} provided the  comparison functor $K: \Com B \to \Com P^{\pi^*\circ\pi_!}$ is an equivalence of categories. With this last definition full contents of Definition~\ref{def.main} is explained. Note that the unit and counit of adjunction $\pi_! \dashv \pi^*$, the associated monad and comparison functor are translations of structures described in Section~\ref{sec.clas}. This might be interpreted as passing of a consistency check and as confirmation of faithfulness of the translation procedure employed here.

\section{Identifying principal comodule algebras}\label{sec.pcalg}
\setcounter{equation}{0}
In this section we would like to work out what noncommutative $H$-principal bundles are in the braided (symmetric) monoidal category opposite to the category of vector spaces. Fix a field $\k$.  As explained in Section~\ref{sec.quant}, the category of $\k$-vector spaces $\vect_\k$ is a monoidal category with the usual tensor product as the monoidal product and with the field $\k$ as the monoidal unit. The flip $V\ot W\ni v\ot w\mapsto w\ot v\in W\ot V$ is a braiding (in fact symmetry as it squares to identity) in $\vect_\k$.  The category of vector spaces has all equalisers and coequalisers and they are preserved by the tensor product.

We will study noncommutative principal bundles in the category $\vect_\k^{op}$, {\bf opposite} to $\k$-vector spaces. $\vect_\k^{op}$ has $\k$-vector spaces as  objects and linear transformations as morphisms, but a morphism $f: V\to W$ in $\vect_\k^{op}$ is given by a $\k$-linear transformation $f: W\to V$. Consequently, the composition in  $\vect_\k^{op}$ is given by the opposite composition of linear transformations. This `reversing of arrows' results in `swapping the prefix co-'. We now carefully study the contents of Definition~\ref{def.main} in case $\cB = \vect_\k^{op}$.

A Hopf algebra $H$ in $\vect_\k^{op}$ is also a Hopf algebra in $\vect_\k$, and vice versa (what was the multiplication in one case becomes the comultiplication in the other and vice versa); i.e.\ in both cases $H$ is a standard Hopf algebra over $\k$. A coalgebra $B$ in $\vect_\k^{op}$ the same as a $\k$-algebra $B$. A right $H$-module coalgebra $P$ in $\vect_\k^{op}$ is the same as a right $H$-comodule $\k$-algebra $P$ with coaction $\varrho: P\to P\ot H$ in $\vect_\k$. The diagram \eqref{diag2} of coalgebras in $\vect_\k^{op}$, after inverting all the arrows, becomes the following diagram of $\k$-algebras
\begin{equation}\label{diag4}
\xymatrix{ B\ar[r]^\pi & P
\ar@<0.5ex>[rr]^-{\id_P\ot 1_H}\ar@<-0.5ex>[rr]_-{\varrho} && 
P\ot H ,}
\end{equation}
where $1_H : \k\to H$ is the unit (map) of $H$.
The $\ot$-exactness of \eqref{diag4} means that 
\begin{alist}
\item  \eqref{diag4} is an equaliser diagram of $\k$-linear transformations,
\item \eqref{diag4} is a cokernel pair diagram of $\k$-algebras,
\item  $\pi$ is an effective descent monomorphism of $\k$-algebras.
\end{alist}
Since the equaliser of linear transformations is the same as the kernel of their difference, the condition (A) means that $B$ can be identified with the subalgebra  of $H$-coaction invariants of $P$,
$$
B\cong P^{coH} := \{ x\in P\; |\; \varrho(x) = x\ot 1_H\},
$$
and $\pi$ is the inclusion map. The cotensor product of comodules in $\vect_\k^{op}$ is the same as the tensor product of modules over $\k$-algebras. Thus the statement (B) means that the unique $\k$-linear map $\can : P\ot_B P \to P\ot H$ fitting the following diagram
$$
\xymatrix{ P\ot_B P \ar[rrrr]^-\can &&&& P\ot H \\
P\ot P \ar[u] \ar[rrrru]_{~\hspace{8mm}m_{P\ot H}\circ (\id_P\ot 1_H\ot \varrho)}, &&&& }
$$
is an isomorphism. Here $m_{P\ot H}$ denotes the multiplication in the tensor product algebra, that is, for all $x, y\in P$ and $g,h\in H$, 
$$
m_{P\ot H} (x\ot g \ot y \ot h) = xy\ot gh;
$$
see \eqref{eq.alg}. The map $\can$ can be easily computed (compare the proof of Lemma~\ref{lem.can.H}),
$$
\can = (m_P \ot \id_H )\circ (\id_P \ot_B \varrho), \qquad x\ot_B y\mapsto x\varrho(y).
$$
Therefore, conditions (A) and (B) mean that $P$ is a {\em Hopf-Galois $H$-extension of $B$}. 

The map of $\k$-algebras $\pi: B\to P$ induces a pair of functors between the categories of their right modules $\Mod B$ and $\Mod P$ (these are the categories of right comodules over the comonoids $B$ and $P$ in $\vect_\k^{op}$). The functor 
$
\pi_! : \Mod P \to \Mod B,
$
is the restriction of scalars functor, which views every right $P$-module $E$ as a right $B$-module via the map $\pi$, i.e., for all $a\in B$ and $e\in E$, $e\cdot a := e\cdot \pi(a)$. Equivalently, $\pi_!$ can be described as the homorphism functor $\rhom P{{}_BP} -$ (and in the algebraic geometry literature often denoted by $\pi_*$ as it is the {\em direct image functor}). The functor $\pi^* : \Mod B\to \Mod P$ is the extension of scalars functor, for any right $B$-module $X$, $\pi^*(V) = V\ot_B P$, with the action of $P$ induced by the multiplication in $P$, i.e.\ $(v\ot_B x)\cdot y := v\ot_B xy$, for all $x,y\in P$ and $v\in V$. Equivalently, $\pi^*$ can be described as the tensor functor $-\ot_BP$ (and in the algebraic geometry literature it is often referred to as the {\em inverse image functor}). 

By reversing the arrows in the underlying monoidal category, we reverse adjunctions. Therefore, when viewed from the point of view of vector spaces, the functor $\pi^*$ is the {\bf left} adjoint of $\pi_!=\pi_*$. The formula for the unit of adjunction written in Section~\ref{sec.quant} gives the counit of adjunction $\pi^* \dashv \pi_*$ and vice versa. The composite $T= \pi^*\circ \pi_*$ is a {\bf comonad} on the category $\Mod P$. For all right $P$-modules $E$, $T(E) = E\ot_BP$ with $P$-action $\id_E\ot_B m_P$. The comultiplication and counit of $T$ come out as
$$
E\ot_B P \to E\ot_BP\ot_BP, \qquad e\ot_B x\mapsto e\ot_B 1_P\ot_B x,
$$
and
$$
E\ot_B P \to E, \qquad e\ot_B x\mapsto e\cdot x. 
$$

A {\em comodule} (or {\em coalgebra} in category theory terminology) over the comonad $T$, consists of a right $P$-module $E$ (i.e.\ an object of the category $\Mod P$ on which $T$ operates; the $P$-action is not written explicitly) together with a right $P$-linear map $\xi: T(E) = E\to E\ot_B P$ which satisfies conditions obtained by reversing all the arrows   in diagrams \eqref{alg}  and replacing cotensor product by the tensor product. On elements $e\in E$ and writing $\xi(e) = \sum_i e_i\ot_B x_i \in E\ot_B P$ the conditions satisfied by $\xi$ are
\begin{equation}\label{des}
\sum_i \xi(e_i) \ot_B x_i = \sum_i e_i\ot_B1_P\ot_B  x_i , \qquad \sum_i e_i\cdot x_i =e.
\end{equation}
Pairs $(E,\xi)$ with $\xi$ satisfying conditions \eqref{des} are known as {\em descent data} and form the backbone of descent theory of (noncommutative) algebras; see the classic text \cite{Gro:techI} or  modern elegant expositions \cite[Section~4.7]{Bor:han2}, \cite{Nus:non}. The category of descent data corresponding to the $\k$-algebra map $\pi$ is denoted by $\Desc_\pi$. We have just shown that the category of comodules over $T$ can be identified with  $\Desc_\pi$, and therefore there is the following triangle of categories of functors
$$
\xymatrix{ \Mod B\ar[rr]^{K}\ar[dr]_{\pi^*} & & \Desc_\pi \ar[dl]^{U} \\
& \Mod P .& }
$$
For all right $B$-modules $V$, the comparison functor $K$ returns the descent datum $(V\ot_B P, \xi)$, where $\xi: V\ot_B P\to V\ot_B P\ot_B P$, $v\ot_B x\mapsto v\ot_B 1_P\ot_B x$. The functor $K$ is an equivalence if and only if the algebraic descent associated to $\pi$ is effective, which, by the Grothendieck theorem (extended in \cite{Cip:dis}) is equivalent to the statement that $P$ is {\em faithfully flat} as a left $B$-module. (Recall that, by definition, $P$ is faithfully flat as a left $B$-module if the sequence of right $B$-module maps
$
\xymatrix{V\ar[r]^-f &  V' \ar[r]^-g & V''}
$
is exact if and only if the sequence
$
\xymatrix{V\ot_B P\ar[rr]^-{f\ot_B\id_P} &&  V'\ot_B P \ar[rr]^-{g\ot_B\id_P} && V''\ot_B P}
$
is exact.) 

Therefore, a noncommutative $H$-principal bundle in $\vect_\k^{op}$ is the same as a Hopf-Galois $H$-extension $B\subseteq P$ such that $P$ is faithfully flat as a left $B$-module. If $H$ has a bijective antipode, then  by \cite[Theorems~4.8~\&~5.6]{SchSch:gen}, $P$ is an {\em $H$-equivariantly projective} left $B$-module, that is the restriction of the multiplication map $B\ot P\to P$ has a left $B$-module right $H$-comodule section (splitting). An $H$-equivariantly projective Hopf-Galois extension is termed a {\em principal $H$-comodule algebra} in \cite{HajKra:pic}. A strong evidence that principal comodule algebras should be understood as principal bundles in noncommutative geometry  is being uncovered in \cite{BauHaj:pet}. 
We have just arrived at the following

\begin{ident}
Let $H$ be a Hopf algebra (over a field $\k$) with bijective antipode. Then principal $H$-comodule algebras can be identified with noncommutative $H$-principal bundles in the braided monoidal category $\vect_\k^{op}$.
\end{ident}

With this the main aim of these notes is achieved.

\section{Further directions}\label{sec.dir}
\setcounter{equation}{0}
In this final section we comment on some aspects of Definition~\ref{def.main} in a general braided monoidal category $\cB$ (with equalisers preserved by the tensor product). Slightly greater familiarity with category theory than in the rest of this article is required here. The discussion of internal categories in \cite[Chapter~8]{Bor:han2} or \cite[Chapter~2]{Joh:top} and of Beck's theorem in \cite[Section~3.3]{BarWel:top} might be of some assistance. 

We consider first a monoidal category $\cB$ (not necessarily braided) with coequalisers preserved by the tensor product).  Take a morphism of comonoids $\pi: P\to B$ and consider the associated monad $T$ defined by equations \eqref{T}--\eqref{mu} and the comparison functor $K: \Com B \to \Com P^T$  defined in equation \eqref{comparison}. For any right $P$-comodule $(E,\lambda)$, there is an isomorphism (of $P$-comodules) $E\to E\coten P P$ induced from the coaction $\lambda: E\to E\ot P$. Taking this isomorphism into account, $T$ can be identified with the functor
$$
\hat{T}:= -\coten P P\coten B P : E\mapsto E\coten P P\coten B P.
$$
 Again taking the above isomorphism as a right unitor and a similar isomorphism as a left unitor, the category $\Bicom P$ of $P$-bicomodules is a monoidal category with the cotensor product $\coten P$ as its monoidal product. The fact that $\hat{T}$ is a monad means that $P\coten B P$ is a monoid in $\Bicom P$. The multiplication of this monoid is 
 $$
 \id_P\coten B \pi_B \coten B \id_P: P\coten B P\coten B P \cong P\coten B P\coten P P\coten B P \to P\coten B B\coten B P \cong P\coten B P,
 $$
 and the unit is induced from $\Delta_P$. The category $\Com P^T$ is isomorphic to the category of right modules over $P\coten B P$.  As explained in \cite{Agu:int}, $(P, P\coten BP)$ might be interpreted as an {\em internal category} in the monoidal category $\cB$ ($P$ is the object of objects and $P\coten B P$ is the object of morphisms). With this interpretation in mind, $\Com P^T$ is isomorphic to  the category  $\Pr{P, P\coten BP}$ of {\em internal presheaves} on $(P, P\coten BP)$; see \cite[Chapter~6]{Agu:int} and \cite{Vaz:phd}. As $\Com B$ is the category of right modules over the trivial monoid in  $\Bicom B$, the functor $K$ compares the categories of presheaves over two internal categories in $\cB$.

The necessary and sufficient conditions for $K$ to be an equivalence are provided by the Beck Precise Monadicity  Theorem \cite{Bec:tri}. In particular, if $K$ is an equivalence, then its inverse functor $K^{-1}$ is defined as follows. For all $(E,\lambda,\xi)$ in $\Com P^T$, $K^{-1}(E,\lambda,\xi) = E^P$,
where $E^P$ is defined by the coequaliser
\begin{equation}\label{Beck.coeq}
\xymatrix{ E\coten B P 
\ar@<0.5ex>[rr]^-{\xi}\ar@<-0.5ex>[rr]_-{\epsilon_{(E,\lambda)}} && 
E \ar[rr]^-{\Pi_E} && E^P.}
\end{equation}
The existence of coequaliser \eqref{Beck.coeq} is guaranteed by the fact that $K$ is an equivalence. The natural isomorphism $\Phi: \id_{\Com P^T} \to K\circ K^{-1}$ is the composite, for all $(E,\lambda,\xi)$ in $\Com P^T$,
$$
\Phi_{(E,\lambda,\xi)} : \xymatrix{ E\ar[rr]^-{\hat{\lambda}} && E\coten B P \ar[rr]^{\Pi_E\coten B \id_P} && E^P\coten B P.}
$$
The natural isomorphism $\Psi: K^{-1}\circ K\to \id_{\Com B} $, for any right $B$-comodule $(E,\lambda)$, is the unique filler in the following diagram
$$
\xymatrix{E\coten B P\coten B P \ar@<0.5ex>[rrrr]^-{\epsilon_{(E,\lambda)}\coten B \id_P}\ar@<-0.5ex>[rrrr]_-{\epsilon_{(E\coten BP,\id_E\coten B((\id_E\ot \pi)\circ \Delta_P))}} && &&
E \coten B P \ar[rr]^-{\Pi_{E\coten B P} }\ar[drr]_{\epsilon_{(E,\lambda)}}&& (E\coten B P)^P \ar@{-->}[d]^{\Psi_{(E,\lambda)}} \\
&&&&&& E.}
$$

In the context of Definition~\ref{def.main} the category $\Com P^T$ and the functors $K$, $K^{-1}$ can be described in purely Hopf-algebraic terms.

\begin{proposition}  \label{prop.invariants}
Let $\pi: P\to B$ be a non-commutative principal $H$-bundle in a braided monoidal category $(\cB,\ot,\1,\tau)$ as in Definition~\ref{def.main}. Then:
\begin{zlist}
\item The morphism $\can: P\ot H \to P\coten B P$ described in Lemma~\ref{lem.can.H} is a left $P$-colinear and right $H$-linear map.
\item  The functor $T$ is naturally isomorphic to the functor which sends right $P$-comodule $(E,\lambda)$ to the right $P$-comodule $(E\ot H,\nu)$, where 
$$
\nu = (\id_E\ot \id_H\ot \varrho)\circ (\id_E\ot \tau_{P,H}\ot \id_H)\circ (\lambda \ot \Delta_H).
$$
Upon this isomorphism the resulting functor is a monad with multiplication $\id_E\ot m_H$ and unit $\id_E\ot 1_H$. 
\item The category $\Com P^T$ is isomorphic to the category of relative $[P,H]$-Hopf modules, $\Mod{P,H}$. The objects of  $\Mod{P,H}$ are triples $(E,\lambda ,\zeta)$, where $(E,\lambda)$ is a right $P$-comodule and $(E,\zeta)$ is a right $H$-module, such that
$$
\lambda\circ\zeta = (\zeta\ot \varrho)\circ (\id_E\ot \tau_{P,H} \ot \id_H)\circ (\lambda\ot \Delta_H).
$$
\item The comparison functor $K$ is naturally isomorphic to the functor which sends every right $B$-comodule $(E,\lambda)$ to the relative $[P,H]$-Hopf module 
$$
(E\coten BP, \id_E\coten B \Delta_P, \id_E\coten B \varrho).
$$
\item The functor $K^{-1}$ is naturally isomorphic to the {\em $H$-invariants functor}, which sends $(E,\lambda,\zeta)$ to the $B$-comodule $(E^H, \lambda^H)$, where $E^H$ is the coequaliser of $\zeta$ and $\id_E\ot \eps_H$, and $\lambda^H$ is induced from the composite
$$
\xymatrix{ E \ar[r]^-\lambda & E\ot P \ar[rr]^-{\id_E\ot \pi} && E\ot B \ar[r] & E^H\ot B.}
$$
\end{zlist}
\end{proposition}
\begin{proof}
The morphism \eqref{eq.can} is both left $P$-colinear and right $H$-linear, where  $P\ot H$ (resp.\ $P\ot P$)  is a $P$-comodule by $\Delta_P\ot \id_H$ (resp.\  $\Delta_P\ot \id_P$) and $H$-module  by $\id_P \ot m_H$ (resp.\ $\id_P \ot \varrho$).  Let $\iota : P\coten B P \to P\ot P$ be the equaliser monomorphism, so that  
$$
\iota \circ \can = (\id_P \ot \varrho)\circ (\Delta_P\ot \id_H);
$$
see equation~\eqref{eq.can}. The preservation of equalisers by tensor product allows one to make $P\coten BP$ into a left $P$-comodule with coaction $\Delta_P\coten B \id_P$ defined as the unique filler in the diagram
 $$
\xymatrix{ P\ot (P\coten B P)  \ar[rr]^-{\id_P\ot \iota} && P\ot P\ot P 
\ar@<0.5ex>[rr] \ar@<-0.5ex>[rr] && P\ot P\ot B\ot P\\ 
&& P\coten B P\ar[u]_{(\Delta_P\ot \id_P)\circ\iota}\ar@{-->}[llu]^{\Delta_P\coten B \id_P}; && }
$$
compare diagram \eqref{equaliser}. By construction, $\iota$ is a $P$-colinear map. Similarly, he right $H$-action on $P\coten B P$ is defined as $\id_P\coten B \varrho$, 
$$
\xymatrix{ P\coten B P  \ar[rr]^-{ \iota} && P\ot P
\ar@<0.5ex>[rrrr]^-{(\id_P\ot \pi\ot \id_P)\circ(\Delta_P\ot \id_P)} \ar@<-0.5ex>[rrrr]_-{(\id_P\ot \pi\ot \id_P)\circ(\id_P\ot \Delta_P)} &&&&  P\ot B\ot P\\ 
&& P\coten B P\ot H\ar[u]_{(\id_P\ot \varrho)\circ(\iota\ot H)}\ar@{-->}[llu]^{\id_P\coten B \varrho}; &&& }
$$
compare diagram \eqref{eq.pi*}. Again by construction $\iota$ is an $H$-linear map. 
 Therefore,
$$
( \id_P \ot \iota)\circ ( \id_P\ot \can)\circ (\Delta_P\ot \id_H) = (\Delta_P \ot \id_P)\circ \iota \circ \can 
= (\id_P\ot \iota)\circ  (\Delta_P\coten B \id_P)\circ \can.
$$
The first equality expresses the $P$-colinearity of morphism \eqref{eq.can}, while the second is the $P$-colinearity of $\iota$. Since the tensor product preserves equalisers,
$(\id_P \ot \iota) : (P\ot P)\coten B P\cong P\ot (P\coten B P)\to P\ot P\ot P$ is the equaliser monomorphism. The (left) cancellation property of monomorphisms now yields
$$
( \id_P\ot \can)\circ (\Delta_P\ot \id_H) =(\Delta_P\coten B \id_P)\circ \can,
$$
i.e.\ $\can$ is a left $P$-colinear morphism as required. The $H$-linearity of $\can$ is proven is a similar way. This proves statement (1).

Since $\can$ is an isomorphism of left $P$-comodules and right $H$-modules,  there is a chain of isomorphisms
$$
\xymatrix{E\coten B P\ar[r]^-\cong & E\coten PP\coten B P\ar[rr]^-{\id_E\coten P\can^{-1}} && E\coten P P\ot H \ar[r]^-\cong & E\ot H.}
$$
One uses this composite isomorphism, which is natural in $E$, to form the required isomorphisms of functors and  the required equivalence of categories. All the remaining statements are obtained by translation through this isomorphism. 
\end{proof}

Once this interpertation of categories and functors is made, we find ourselves in the realm of Hopf-Galois theory in braided monoidal categories as developed for example in \cite{Sch:bra}. Making arguments dual to  these in \cite[Section~4]{Sch:bra} one associates to a noncommutative principal $H$-bundle $\pi: P\to B$ a {\em quantum category} in the sense of \cite[Section~12]{DayStr:cat} (or a braided version of a {\em bicoalgebroid} in terminology of \cite[Section~5]{BrzMil:bia}) as follows. $B\cong P^H$ is the object of objects and $G= (P\ot P)^H$ is the object of morphisms, where $P\ot P$ is a $[P,H]$-Hopf module with coaction $\id_P\ot \Delta_P$ and the diagonal right $H$-action. The source and target maps are induced from $\pi\ot \eps_P$ and $\eps_P\ot \pi$. The multiplication (or composition of morphisms) $m_G: G\coten BG \to G$ is induced from the composite
$$
\Pi_{P\ot P}\circ (\varrho\ot \id_P)\circ (\id_P\ot\eps_P\ot \id_H\ot \id_P)\circ (\id_P\ot\can^{-1}\ot \id_P),
 $$
 while the unit $u_G: B\to G$ is induced from $\Pi_{P\ot P}\circ\Delta_P$. The coalgebra structure of $G$ is induced from the tensor product coalgebra $P\ot P^{op}$, where $P^{op}$ is the opposite coalgebra to $P$, that is the object $P$ of $\cB$ with comultiplication $\tau_{P,P}\circ \Delta_P$ and counit $\eps_P$.

 \end{document}